\documentclass[12pt, reqno]{amsart}

\usepackage{latexsym}
\usepackage{amssymb}
\usepackage{mathrsfs}
\usepackage{amsmath}
\usepackage{fancybox,color}
\usepackage{enumerate}
\usepackage[latin1]{inputenc}
\usepackage{eurosym}
\newcommand{\kom}[1]{}
\renewcommand{\kom}[1]{{\bf [#1]}}

 \def\1{\raisebox{2pt}{\rm{$\chi$}}}

\newtheorem{theorem}{Theorem}[section]

\newtheorem{lemma}[theorem]{Lemma}

\newtheorem{remark}[theorem]{Remark}

\newcommand{\N}{{\mathbb N}}
\newcommand{\Z}{{\mathbb Z}}

 \def\1{\raisebox{2pt}{\rm{$\chi$}}}
 
\def\vint_#1{\mathchoice%
          {\mathop{\kern 0.2em\vrule width 0.6em height 0.69678ex depth -0.58065ex
                  \kern -0.8em \intop}\nolimits_{\kern -0.4em#1}}%
          {\mathop{\kern 0.1em\vrule width 0.5em height 0.69678ex depth -0.60387ex
                  \kern -0.6em \intop}\nolimits_{#1}}%
          {\mathop{\kern 0.1em\vrule width 0.5em height 0.69678ex
              depth -0.60387ex
                  \kern -0.6em \intop}\nolimits_{#1}}%
          {\mathop{\kern 0.1em\vrule width 0.5em height 0.69678ex depth -0.60387ex
                  \kern -0.6em \intop}\nolimits_{#1}}}
\def\vintslides_#1{\mathchoice%
          {\mathop{\kern 0.1em\vrule width 0.5em height 0.697ex depth -0.581ex
                  \kern -0.6em \intop}\nolimits_{\kern -0.4em#1}}%
          {\mathop{\kern 0.1em\vrule width 0.3em height 0.697ex depth -0.604ex
                  \kern -0.4em \intop}\nolimits_{#1}}%
          {\mathop{\kern 0.1em\vrule width 0.3em height 0.697ex depth -0.604ex
                  \kern -0.4em \intop}\nolimits_{#1}}%
          {\mathop{\kern 0.1em\vrule width 0.3em height 0.697ex depth -0.604ex
                  \kern -0.4em \intop}\nolimits_{#1}}}

\newcommand{\aveint}[2]{\mathchoice%
          {\mathop{\kern 0.2em\vrule width 0.6em height 0.69678ex depth -0.58065ex
                  \kern -0.8em \intop}\nolimits_{\kern -0.45em#1}^{#2}}%
          {\mathop{\kern 0.1em\vrule width 0.5em height 0.69678ex depth -0.60387ex
                  \kern -0.6em \intop}\nolimits_{#1}^{#2}}%
          {\mathop{\kern 0.1em\vrule width 0.5em height 0.69678ex depth -0.60387ex
                  \kern -0.6em \intop}\nolimits_{#1}^{#2}}%
          {\mathop{\kern 0.1em\vrule width 0.5em height 0.69678ex depth -0.60387ex
                  \kern -0.6em \intop}\nolimits_{#1}^{#2}}}

\newcommand{\A}{\mathcal A}

\begin{document}
\title[%The variation of 
A note about the $\ell^p$-improving property of the average operators
]
{A note about the $\ell^p$-improving property of the average operator
}
\author{ Jos{\'e} Madrid}

\date{\today}

\address{Department of Mathematics, UCLA, 520 Portola Plaza, Los Angeles, CA 90095, USA}
\email{jmadrid@math.ucla.edu}

\maketitle
%{\small \textbf{Abstract.} 
\begin{abstract}In this paper we give a short proof of the $\ell^p$-improving property of the average operator along the square integers and more general quadratic polynomials.  Moreover we obtain a similar result for some higher degree polynomials. We also show an elementary proof of the $\ell^p$-improving property of the average operator along primes.

\end{abstract}

\section{Introduction}

%%%%%%%%%%%%%%%%%%%%%%%%%%%%%%%%%%5555

%\begin{conjecture}\label{poly}  
For $f\in \ell^2(\Z)$. Define the average of  $f$ along the polynomial $P$ mapping the integers to the integers, by 
\begin{align}
A^{P}_{N} f(x)&:=\frac{1}{N}\sum_{k=1}^N f(x+P(k)), 
\end{align}
In the case $P(x)=x^d$ we will denote this average by $A^{d}_{N}f$ for $d>2$ and $A_Nf$ for $d=2$. Along this paper $\|f\|_p$ denotes the $\ell^p-$norm of the function $f:\Z\to\mathbb{R}$ and $p'=\frac{p}{p-1}$.\\

We prove that if $ 3/2 < p \leq 2$ and $P(x)=ax^2+bx+c\in \Z[X]$ is a quadratic polynomial with no negative coefficients, then $ A^P_Nf$ satisfies an $ \ell ^{p}$-improving estimate: 
\begin{equation*}
 N ^{-2/p'} \lVert A^P_N f \rVert _{ p'} \lesssim N ^{-2/p} \lVert f\rVert _{{p}},  
\end{equation*}
for every $f:\Z\to\mathbb{R}$. %supported in $[0,aN^2+bN+\frac{b^2}{4a}]$. 
The range $(3/2,2]$ is optimal. We are able to extend this result for the higher degree polynomials $P(x)=x^d$, however, so far we can do that only in a (probably) not optimal range. We also obtain a similar improving result for averages along primes through our method.\\ 

The regularity on $\ell^p$-spaces of the averages operator along the square integers  was originally studied by Bourgain in \cite{B}. The $l^{p}$-improving estimates for compact supported functions were recently studied in \cite{HLY}, analyzing a good approximation for the corresponding multiplier, coming from the Hardy Littlewood circle method. The argument in that paper 
use strongly the fact that $P(x)=x^2$ and it does not extend even for the case $p(x)=x^2+x$ as pointed out by the authors, due to some difficulties coming from the minor arcs, in that paper they asked about wether or not we continue having the $\ell^p$-improving estimate for other quadratic polynomials like $x^2+x$, and for higher degree polynomials (see conjecture 6.3 in \cite{HLY}). 
In this paper we give a positive answer to that question for $d=2$ and in the higher degree case for $P(x)=x^d$ in the the range $(2-4/(2+d(2^d+2)),2]$ through a complete different argument, in particular we recover the Theorem 1.1 in \cite{HLY}. The main ingredient is to find a relation between the average operator and the discrete fractional integral operator in order to use the results from \cite{SW} and \cite{Pi}. This method also works for the average operator along primes as discussed below. Variants of this were studied in \cite{K},\cite{Pi2} and \cite{SW2}, for some other interesting related results see \cite{H}, \cite{HKLY}, \cite{K},\cite{KL},\cite{MST}.

\begin{theorem}\label{case $x^2$}
For every $3/2< p\leq 2$ there is a constant $C_p>0$, such that for all $N\in\N$ %supported in $I=[-N^2,N^2]$ 
we have that
$$
\|A_{N}f\|_{p'}\leq C_p N^{2/p'-2/p}\|f\|_p,
$$
for every function $f:\Z\to\mathbb{R}$
\end{theorem}

The next results is an extension of this theorem. %to any quadratic polynomial.

\begin{theorem}\label{general}
Let $P(x)=ax^2+bx+c\in \mathbb{Z}[x]$ be a quadratic polynomial with no negative coefficients. For every $3/2< p\leq 2$ and $N\in\N$, there is a constant $C_p>0$, such that for all $N\in\N$ %supported in $I=[-aN^2-bN-\frac{b^2}{4a},aN^2+bN+\frac{b^2}{4a}]$ 
we have that
$$
\|A^P_{N}f\|_{p'}\leq C_p\left(2a+\frac{b}{N}\right)(2aN+b)^{2/p'-2/p}\|f\|_p,
$$
for every function $f:\Z\to\mathbb{R}$.
\end{theorem}
In particular, if $P(x)=x^2$, we have that $a=1, b=0$ and $c=0$, we recover the previous theorem.

For all $d>2$ we define
$$
\widetilde p_d=2-\frac{4}{2+d(2^{d}+2)}.
$$

The next theorem establish the desired result when $p$ is close to $2$, more precisely when $p>\widetilde p_d$. 
\begin{theorem}\label{case $x^d$}
For every\ \ $\widetilde p_{d}< p\leq 2$, there is a constant $C=C(p,d)>0$, such that for all $N\in\N$  %supported in $I=[-N^d,N^d]$ 
we have that
$$
\|A^{d}_{N}f\|_{p'}\leq C N^{d/p'-d/p}\|f\|_p,
$$
for every function $f:\Z\to\mathbb{R}$.
\end{theorem}

\begin{remark}
Using Theorem 1 in \cite{Pi} we can go slightly lower than $\widetilde p_d$ for $d>11$. %However we are unable to solve the conjecture \ref{} for $d>2$ in full generality.
\end{remark}

Our final result discuss the improving property of the average operator along primes. This result was recently established in \cite{HKLY} in this paper we present an elementary proof of this fact.

For every $f:\Z\to\mathbb{R}$ we define the average operator along primes to be
$$
\A_{N}f(x)=\frac{1}{N}\sum_{p\leq N}{f(x-p)}\log p
$$ 
where the sum is taken over all primes with size at most $N$.

\begin{theorem}\label{averages along primes}
Let $1<p\leq 2$, then there exists a constant $C_p>0$ such that for all $N\in\N$ we have that
\begin{equation}\label{ineq along primes}
\|\A_Nf\|_{p'}\leq C_pN^{\frac{1}{p'}-\frac{1}{p}}\|f\|_p,
\end{equation}
for every $f:\Z\to\mathbb{R}$.
\end{theorem}
\begin{remark}
In the case $p>2$, if $f$ is supported in $[0,N]$ the inequality \eqref{ineq along primes} follows immediately as a consequence of the $\ell^p$-boundedness of the operator $\A_N$ established by Bourgain in \cite{B2} and H\"older inequality.
\end{remark}

%The argument in \cite{HLY} does not extend for the case $d>2$ due to some difficulties coming from the minor arcs. %We are very optimistic about the efficiensy of this new approach to deal with the situation 

\section{Preliminaries}

 We write $A\lesssim B$ if there exists an absolute constant $C$ such that  $A\leq CB$. If the constant depends on parameter $\lambda$  we denote that with a subscript, such as $A\lesssim _{\lambda}B$. We write $A\sim B$ if both $A\lesssim B$ and $B\lesssim A$. $\delta_k$ denotes the classical Dirac delta function supported at the point $k$, more precisely $\delta_k(k)=1$ and $\delta_{k}(n)=0$ for all $n\neq k$.

%We denote by $d(n)$ the divisor functions (counting the number of divisors of $n$), an important element of the proof will be the wellknown estimative for the divisor function: for every $\epsilon>0$ there exists a constant $C_\epsilon$ such that
%\begin{equation}\label{d(n)}
%d(n)\leq C_{\epsilon}n^{\epsilon} \ \text{for every}\ n\in\N.
%\end{equation}
Let us focus in the case $P(x)=x^d$.
Let $ d \geq 2 $ be an integer, we start observing that %in order 
%to have the $l^p$-improving property 
an inequality like
\begin{equation}\label{lp improving}
 N ^{-d/p'} \lVert A^{d}_N f \rVert _{ p'} \lesssim N ^{-d/p} \lVert f\rVert _{{p}},  
\end{equation}
for functions $f:\Z\to\mathbb{R}$ supported in $[-N^{d},N^d]$ along the polynomial $P(x)=x^d$ would be optimal in terms of magnitude, in fact, it is enough to consider $f=\chi_{[-N^d,N^d]}$, in this case $\|f\|_p=2^{1/p}N^{d/p}$ while $\|A^{d}_Nf\|_{p'}\geq \|A^{d}_{N}f\|_{\ell^{p'}[-N^d,0]}=N^\frac{d}{p'}$.

Moreover, we observe, that in order to have the $\ell^p$-improving property \eqref{lp improving}, the condition $p\geq 2-\frac{1}{d}=:p_d$ is necessary. In fact, if $f=\delta_0$, then $\|f\|_{p}=1$ and $\|A^{d}_{N}f\|_{p'}=\left(N\frac{1}{N^{p'}}\right)^{\frac{1}{p'}}=\frac{1}{N^{1/p}}$, then \eqref{lp improving} holds for $f=\delta_0$ only if $N^{-d/p'}N^{-1/p}\lesssim N^{-d/p}$ which is equivalent to have $\frac{d-1}{p}\leq \frac{d}{p'}$ {\it{i.e}} $p>p_d$.
In particular if $d=2$ then $p_2=3/2$ is the endpoint, and $p'_2=3$.\\

\subsection{Structure of the paper.} Section 3 contains the proof of our main theorems for averages operators along polynomials. In section 4 we present the proof of our Theorem \ref{averages along primes} for averages operators along primes.

\section{proof of main results}
%In order to establish the Theorem \ref{case $x^2$}, it is enough to prove that
%\begin{proposition}\label{Main prop}
%There exists a constant $C>0$ such that for any function $f:\Z\to\mathbb{R}$ supported in $[0,N^2]$ we have
%\begin{equation}\label{key}
%\|A_Nf\|_{3}\leq CN^{-2/3}\|f\|_{3/2}.
%\end{equation}
%\end{proposition}
%\begin{remark}
%The inequality \eqref{key} is optimal, in fact if $f=\delta_N$ then $\|A_Nf\|_{3}\sim N^{-2/3}\|f\|_{3/2}$.
%\end{remark}

%Assuming Proposition \ref{Main prop}, Theorem \ref{case $x^2$} will follow via interpolation.
Now we are in position to start proving our main theorems, a key idea will be to relate the average operator with the discrete fractional integral operator.
For $f:\Z\to\mathbb{R}$ we define 
$$
I_{d,\lambda}f(n)=\sum_{m=1}^{\infty} \frac{f(n-m^d)}{m^\lambda}
$$
where $0<\lambda<1$ and $d\geq 1$ is an integer. We denote by $I_{\lambda}f=I_{2,\lambda}f$. This is the well knonwn discrete fractional integral operator.

\subsection{\it{Case d=2, $P(x)=x^2$}.}
\begin{proof}[Proof of Theorem \ref{case $x^2$}]
We start observing that by Young inequality
\begin{equation}\label{young ineq}
\|A_Nf\|_{2}=\|K_N*f\|_2\leq \|K_N\|_1\|f\|_2=\|f\|_2=N^{2/2-2/2}\|f\|_2
\end{equation}
where $K_N=\frac{1}{N}\sum_{k\leq N}\delta_{-k^2}$. %Assuming \eqref{key} we also have
%$$
%\|A_Nf\|_{3}\leq CN^{2/3-4/3}\|f\|_{3/2}.
%$$
%for some constant $C>0$. Then, by interpolation, we have that
%$$
%\|A_Nf\|_{p'}\leq  C_pN^{2/p'-2/p}\|f\|_{p}
%$$
%for all $p\in[3/2,2]$.
%\end{proof}
%\begin{proof}[Proof of Proposition \ref{Main prop}]
%%%%%%%%%%%%%%%%%%%%new approach
Then we can focus in the case when $p\in (3/2,2)$, we define 
$$
\lambda=1-\left(\frac{2}{p}-\frac{2}{p'}\right).
$$
Thus: $0<\lambda<1$ and
$
\frac{1-\lambda}{2}=\frac{1}{p}-\frac{1}{p'}.
$
Moreover, since $p>\frac{3}{2}$ we have that
$
2>\frac{1}{p-1}=\frac{p'}{p}.
$
Then
$$
\frac{1}{p}>\frac{2}{p}-\frac{2}{p'}=1-\lambda\ \ \text{and} \ \frac{1}{p'}<\lambda.
$$
%Moreover, there exists $\delta>0$ such that if $p-3/2<\delta$, then  
%$$
%\lambda=1-\left(\frac{4}{p}-2\right)=3-\frac{4}{p}<3-\frac{4}{(3/2)}+\frac{1}{7}<\frac{1}{2}
%$$
Then, as a consequence of Theorem A and Theorem 1 in \cite{SW} we know that there exsits a constant $C=C(p)$ such that
$$
\| I_{\lambda}h\|_{p'}\leq C\|h\|_{p},
$$
for every $h\in \l^{p}$. Then we observe that
\begin{eqnarray*}
A_{N}f(x)&:=&\frac{1}{N}\sum_{k\leq N}f(x+k^2)\\
&\leq& N^{\lambda-1}\sum_{k\leq N}\frac{f(x+k^2)}{k^{\lambda}}\\
&\leq& N^{\lambda-1}I_{\lambda}g(-x)
\end{eqnarray*}
for all $x\in\Z$, where $g$ is given by $g(y):=f(-y)$ for every $y\in\Z$. Using this we obtain
\begin{eqnarray*}
\|A_{N}f\|_{p'}\leq N^{\lambda-1}\|I_{\lambda}g\|_{p'}\leq CN^{\lambda-1}\|g\|_{p}= CN^{\lambda-1}\|f\|_{p}.
\end{eqnarray*}
Therefore
$$
\|A_Nf\|_{p'}\leq CN^{2/p'-2/p}\|f\|_p.
$$

%We start observing that
\end{proof}

\subsection{\it{General case $d=2$.}}

Now we are in position to extend this $\ell^p$-improving property to any quadratic polynomial.

\begin{proof}[Proof of Theorem \ref{general}]
We define $g:\Z\to\mathbb{R}$ by
$$
g(4am)=f(m) \ \text{for all} \ m\in\Z \ \text{and} \ g(n)=0 \ \text{if}\ 4a\nmid n. 
$$
We observe that since $f$ is supported in $[-(aN^2+bN+\frac{b^2}{4a}),aN^2+bN+\frac{b^2}{4a}]$ then we have that $g$ is supported in $[-(2aN+b)^2,(2aN+b)^2]$. Therefore
\begin{eqnarray*}
A^P_{N}f(x)&=&\frac{1}{N}\sum_{n\leq N}f(x+an^2+bn+c)\\
&=&\frac{1}{N}\sum_{n\leq N}g(4ax+4a^2n^2+4abn+4ac)\\
&=&\frac{1}{N}\sum_{n\leq N}g(4a(x+c)-b^2+(2an+b)^2)\\
&\leq&\frac{(2aN+b)}{N}\frac{1}{2aN+b}\sum_{k\leq 2aN+b}g(4a(x+c)-b^2+k^2)\\
&=&\left(2a+\frac{b}{N}\right)A_{2aN+b}g(4a(x+c)-b^2).
\end{eqnarray*}
Using the previous calculus, as a consequence of Theorem \ref{case $x^2$} we obtain
\begin{eqnarray*}
\|A^P_{N}f\|_{p'}&\leq&\left(2a+\frac{b}{N}\right)\|A_{2aN+b}g\|_{p'}\\
&\leq& \left(2a+\frac{b}{N}\right)(2aN+b)^{2/p'-2/p}C_p\|g\|_p\\
&\leq& \left(2a+\frac{b}{N}\right)(2aN+b)^{2/p'-2/p}C_p\|f\|_p,
\end{eqnarray*}
for every $p\in(3/2,2]$.
\end{proof}

%%%%%%%%%%%%%%%%%%%%%%%%%%%%%%%%%%%%%%%%%%%%%%%%%%%%%%%%
\subsection{\it{Case $d>1$.}}
We  will adapt the idea used in the proof of Theorem 1.
\begin{proof}[Proof of Theorem \ref{case $x^d$}]
The case $p=2$ follows from Young inequality similarly to \eqref{young ineq}. Let $p\in (\widetilde p_d,2)$, we define 
$
\lambda=1-\left(\frac{d}{p}-\frac{d}{p'}\right).
$
Thus: $0<\lambda<1$ and
$
\frac{1-\lambda}{d}=\frac{1}{p}-\frac{1}{p'}.
$
Moreover, using that $p>p_d$ we see 
$
\frac{1}{p}> 1-\lambda\ \ \text{and} \ \frac{1}{p'}<\lambda.
$
%Moreover, there exists $\delta>0$ such that if $p-3/2<\delta$, then  
%$$
%\lambda=1-\left(\frac{4}{p}-2\right)=3-\frac{4}{p}<3-\frac{4}{(3/2)}+\frac{1}{7}<\frac{1}{2}
%$$
As a consequence of Theorem 1 in \cite{Pi} we know that if $\lambda>\lambda_d$ (this condition is equivalent to $p>\widetilde p_d$) then there exsits a constant $C=C(p,d)$ such that
$$
\| I_{d,\lambda}h\|_{p'}\leq C\|h\|_{p},
$$
for every $h\in \l^{p}$, where 
$
\lambda_{d}:=1-\frac{1}{2^{d-1}+1}. 
$

Then we observe that
\begin{eqnarray*}
A^{d}_{N}f(x)\leq N^{\lambda-1}I_{k,\lambda}g(-x)
\end{eqnarray*}
for all $x\in\Z$, where $g$ is given by $g(y):=f(-y)$ for every $y\in\Z$. Using this we obtain
\begin{eqnarray*}
\|A^{d}_{N}f\|_{p'}\leq N^{\lambda-1}\|I_{\lambda}g\|_{p'}\leq CN^{\lambda-1}\|f\|_{p},
\end{eqnarray*}
which is the desired result.

%We start observing that
\end{proof}

%%%%%%%%%%%%%%%%%%%%%%%%%%%%%%%%%%%%%%%%%%%%%%%%%%%%%%%%%%
\section{Averages along primes}
For $N\in\N$ and $\lambda\in(0,1)$ we define the truncated fractional integral operator along primes to be
$$
J_{\lambda,N} f(x)=\sum_{p\leq N}\frac{f(x-p)}{p^{\lambda}}\log p.
$$
for every $f:\Z\to\mathbb{R}$, where the sum is taken over all the primes $p$ with size at most $N$.

\begin{lemma}\label{fraccionario along primes}
Let $\lambda\in(0,1)$ and $C>0$ be a constant. Assume that $1<p<\frac{1}{1-\lambda}$ and $\frac{1}{q}=\frac{1}{p}-(1-\lambda)$. Then %$J_{\lambda}$ is bounded from $l^{p}$ to $l^{q}$  
\begin{equation*}
\sup_{\alpha\leq CN^{-1/q}\|f\|_p} \alpha^{q}|\{x;J_{\lambda,N} f(x)>\alpha\}|
\lesssim \|f\|^q_p .
\end{equation*}
\end{lemma}

%\begin{corollary}
%a
%\end{corollary}

%\begin{corollary}
%a
%\end{corollary}

\begin{proof}
We assume with loss of generality that $f\geq0$. We start enumerating the primes with size at most $N$.

%and $\|f\|_p=1$.
 %We will prove that $J_\lambda$ is of $(p,q)$ weak type for every $1<p<\frac{1}{1-\lambda}$. So we would like to prove that there exists a constant $C_{p,q}>0$ such that 
%\begin{eqnarray*}
%\alpha|\{x;J_{\lambda}f(x)>\alpha\}|^{\frac{1}{q}}\leq C_{p,q}\|f\|_p \ \ \text{for all}\ \alpha>0.
%\end{eqnarray*}

%To this end, we decompose $J_\lambda$.
%\begin{eqnarray*}
%J_\lambda=\sum_{p}\frac{f(x-p)}{p^{\lambda}}\log p&=&\sum_{n\leq r}\frac{f(x-p_n)}{p_n^{\lambda}}\log p_n + \sum_{r<n}\frac{f(x-p_n)}{p_n^{\lambda}}\log p_n\\
%&=& J'_\lambda+J''_{\lambda}.
%\end{eqnarray*}
%where $r$ will be choose later. \\

%{\it{Analysis of $J'_\lambda$:}}

$$
\{p\leq N; p \ \text{prime}\}:=\{p_1,p_2,\dots,p_{r_N}\}.
$$
We recall that
$$
p_n\sim n\log n \  \text{for all}\ n,
$$
more precisely $n\log n+n\log\log n-n\leq p_n\leq n\log n+n\log\log n$.

We observe that
\begin{equation*}
\|J_{\lambda,N} f\|_p\leq \left(\sum_{n\leq r_N}\frac{\log p_n}{p_n^\lambda}\right)\|f\|_p
\end{equation*}
and
\begin{align*}
\sum_{n\leq r_N}\frac{\log p_n}{p_n^\lambda}%\lesssim \sum_{n\leq r}\frac{\log (n\log n)}{(n\log n)^{\lambda}}
&\lesssim \sum_{n\leq r_N}\frac{\log n}{n^{\lambda}(\log n)^\lambda}\\
&\leq \sum_{n\leq r_N}\frac{(\log r_N)^{1-\lambda}}{n^{\lambda}}\\
&\lesssim (r_N\log r_N)^{1-\lambda}\\
&\lesssim N^{1-\lambda}.
\end{align*}
Therefore
\begin{align*}
|\{x;J_{\lambda,N} f(x)>\alpha/2\}|&\lesssim \frac{1}{\alpha^p}\|J_{\lambda,N} f\|^{p}_p\\
&\leq \frac{1}{\alpha^p}\left(\sum_{n\leq r}\frac{\log p_n}{p_n^\lambda}\right)^p\|f\|^p_p\nonumber \\
&\lesssim \frac{1}{\alpha^p}N^{p(1-\lambda)}\|f\|^p_p\\
&\lesssim \frac{1}{\alpha^p}\alpha^{-qp(1-\lambda)}\|f\|^{q}_p\\
&=\frac{1}{\alpha^q}\|f\|^{q}_p
\end{align*}
for every $\alpha\leq CN^{-1/q}\|f\|_p$.

%{\it{Analysis of $J''_\lambda$}:} As a consequence of H\"older inequality we have that
%\begin{equation*}
% J''_\lambda f(x)=\sum_{r<n}\frac{f(x-p_n)}{p_n^{\lambda}}\log p_n\leq \left(\sum_{r<n}\frac{(\log p_n)^{p'}}{p^{\lambda p'}_n}\right)^{\frac{1}{p'}}\|f\|_p
%\end{equation*}
%for every $x\in\Z$. Moreover
%\begin{eqnarray*}
%\sum_{r<n}\frac{(\log p_n)^{p'}}{p^{\lambda p'}_n}\geq C\sum_{r<n}\frac{(\log n)^{p'}}{(n\log n)^{\lambda p'}}\leq \frac{1}{(\log r)^{(\lambda-1)p'}}\sum_{r<n}
%\end{eqnarray*}

\end{proof}

\begin{proof}[Proof of Theorem \ref{averages along primes}]
Let $p\in(1,2)$. We start observing that by H\"older inequality and the prime number theorem
$$
\A_{N}f(x)\leq \left(\frac{\log N}{N}\right)^{\frac{1}{p}}\|f\|_p\leq \frac{C}{N^{\frac{1}{p'}}}\|f\|_p.
$$
for some constant $C=C_{p}$. Moreover
$$
\A_N f(x)\leq N^{\lambda-1}J_{\lambda,N}f(x) \ \text{for all}\ \ x\in\Z.
$$
where $\lambda:=1-\left(\frac{1}{p}-\frac{1}{p'}\right)$. Then, using the previous lemma we obtain
\begin{align*}
\sup_{\alpha>0}\alpha^{p'}|\{x;\A_N f(x)>\alpha\}|&=\sup_{\alpha\leq CN^{-1/p'}\|f\|_p} \alpha^{p'}|\{x;\A_N f(x)>\alpha\}|\\
&\leq \sup_{\alpha\leq CN^{-1/p'}\|f\|_p} \alpha^{p'}|\{x;J_{\lambda,N} f(x)>\alpha/N^{\lambda-1}\}|\\
&\lesssim N^{p'(\lambda-1)}\|f\|^{p'}_p.
\end{align*}
Therefore 
\begin{equation}
\sup_{\alpha>0}\alpha^{}|\{x;\A_N f(x)>\alpha\}|^{\frac{1}{p'}}\lesssim N^{\frac{1}{p'}-\frac{1}{p}}\|f\|_p.
\end{equation}
This means that $\A_N$ is of weak type $(p,p')$ for every $p\in(1,2)$, then as a consequence of the Marcinkiewicz interpolation theorem we conclude that
$$
\|\A_N f\|_{p'}\lesssim N^{\frac{1}{p'}-\frac{1}{p}}\|f\|_p.
$$
for all $p\in(1,2)$. The case $p=p'=2$ is easier, this follows as consequence of Young inequality
\begin{equation*}
\|\A_Nf\|_{2}=\|K_N*f\|_2\leq \|K_N\|_1\|f\|_2=\|f\|_2=N^{1/2-1/2}\|f\|_2
\end{equation*}
where $K_N(x):=\frac{1}{N}\sum_{p\leq N}\delta_{-p}(x)\log p$ for every $x\in\Z$.
\end{proof}

%%%%%%%%%%%%%%%%%%%%%%%%%%%%%%%%%%%%%%%%%%%%%%

\section{Acknowledgments}
\noindent The author is thankful to Ben Krause for introducing him to this topic. The author is also thankful to Terence Tao for helpful discussions.

%H.L acknowledges M. Parviainen, J. Kinnunen and the Academy of Finland %(project *) 
%for the financial support. J.M. acknowledges J. Kinnunen, Aalto University and Academy of Finland for the support. The authors are thankful to Juha Kinnunen for helpful discussions and guidance during the preparation of this manuscript. The authors thank Emanuel Carneiro for suggesting to think about this problem. The authors also acknowledge the referee for the valuable comments and suggestions.


\begin{thebibliography}{01}
%MPR10b

%\bibitem[AlPe]{AlPe}
%J.M. Aldaz and J. P\'{e}rez L\'{a}zaro.
%\newblock Functions of bounded variation, the derivative of the one-dimensional maximal function, and applications to inequalities
%\newblock{\em Trans. Amer. Math. Soc. 359 (2007),} no. 5, 2443-2461.


\bibitem{B}
J. Bourgain,
\newblock On pointwise ergodic theorems for arithmetic sets, 
\newblock{\em Inst. Hautes Etudes Sci. Publ. Math.
69 (1989), 5-45.} 397-402. MR916338

\bibitem{B2}
J. Bourgain,
\newblock On the maximal ergodic theorem for certain subsets of the integers,
\newblock{\em Israel J. Math. 61 (1988), no. 1, 39-72. MR937581}

\bibitem{H}
K. Hughes,
\newblock $l^p$-improving  for  discrete  spherical  averages,
\newblock{https://arxiv.org/pdf/1804.09260.pdf}

\bibitem{HLY}
R. Han, M. T. Lacey, and F. Yang,
\newblock Averages along the square integers: $l^p-$improving ans sparse inequalities,
\newblock{\em https://arxiv.org/pdf/1907.05734.pdf}.

\bibitem{HKLY}
R. Han, B. Krause, M. T. Lacey, and F. Yang,
\newblock Averages along the primes: improving and sparse bounds,
\newblock{https://arxiv.org/pdf/1909.02883.pdf}

\bibitem{K}
J. Kim
\newblock On discrete fractional integral operators and related Diophantine equations,
\newblock{Math. Res. Lett. 22 (2015), 841-857} 


\bibitem{KL}
 R. Kesler and M. T. Lacey,
\newblock $l^p$-improving  inequalities  for  Discrete  Spherical  Averages,
\newblock{https://arxiv.org/pdf/1804.09845.pdf}

\bibitem{K}
B. Krause,
\newblock Discrete Analogoues in Harmonic Analysis: Maximally Monomially Modulated Singular Integrals Related to Carlesons  Theorem
\newblock{https://arxiv.org/pdf/1803.09431.pdf}


\bibitem{MST}
M. Mirek, E. M. Stein, and B. Trojan,
\newblock $l^p\Z^d$-estimates  for  discrete  operators  of  Radon  type:
variational estimates,
\newblock{ Invent. Math. 209 (2017), no. 3, 665-748. MR368139}


\bibitem{Pi}
L. B. Pierce,
\newblock On discrete fractional integral operators and mean values of Weyl sums,
\newblock{Bull. London Math. Soc., 43 (2011), 597-612. MR2820148}

\bibitem{Pi2}
L. B. Pierce,
\newblock Discrete fractional Radon transforms and quadratic forms,
\newblock{ Duke Math. J. 161 (2012), no. 1, 69-106. MR2872554}


\bibitem{SW}
E. M. Stein and S. Wainger,
\newblock Two discrete fractional integral operators revisited,
\newblock{J. Anal. Math., 87 (2002), 451-479.} Dedicated to the memory of Thomas H. Wolff. MR1945293

\bibitem{SW2}
E. M. Stein and S. Wainger,
\newblock Discrete analogues in harmonic analysis. II. Fractional integration,
\newblock{ J. Anal. Math. 80 (2000), 335-355. MR1771530}



%\bibitem[CMP]{CMP}
%E. Carneiro, J. Madrid and L. B. Pierce,
%\newblock Endpoint Sobolev and BV Continuity
%for Maximal Operators,
%\newblock{\em{J. Funct. Anal. 273 (2017), no. 10,}} 3262--3294.

%\bibitem[CaMo]{CaMo}
%E. Carneiro and D. Moreira, 
%\newblock On the regularity of maximal operators,
%\newblock {\em Proc. Amer. Math. Soc. 136 (2008),} no. 12, 4395--4404.

%\bibitem[CaSv]{CaSv}
%E. Carneiro and B. F. Svaiter, 
%\newblock On the variation of maximal operators of convolution type,
%\newblock {\em J. Funct. Anal. 265 (2013),} 837--865.







\end{thebibliography}
\end{document}